\documentclass[11pt]{amsart}
\usepackage{graphicx}
\usepackage{amsfonts,amsmath,amssymb,amsthm}
\usepackage{mathtools}
\usepackage[a4paper,margin=1.5in]{geometry} 
\usepackage{hyperref}
\DeclarePairedDelimiter\norm{\lVert}{\rVert}

\newcommand{\vbf}{\mathbf{v}}

\newcommand{\ubf}{\mathbf{u}}

\newtheorem{theorem}{Theorem}
\newtheorem{lemma}[theorem]{Lemma}

\newtheorem{conjecture}[theorem]{Conjecture}

\theoremstyle{definition}
\newtheorem{definition}[theorem]{Definition}

\theoremstyle{remark}

\title{Nine and ten lonely runners}
\author[Tanupat Trakulthongchai]{Tanupat Trakulthongchai}
\address{St John's College, University of Oxford, Oxford OX1 3JP, United Kingdom}
\email{tanupat.trakulthongchai@sjc.ox.ac.uk}
\date{\today}

\begin{document}
\begin{abstract}
    The Lonely Runner Conjecture of Wills and Cusick states that if $k+1$ runners start running at distinct constant speeds around a unit-length circular track, then for each runner there is a time when he/she is at least $1/(k+1)$ away from all other runners.  Rosenfeld recently obtained a computer-assisted proof of the conjecture for $8$ runners.  By refining his approach with a sieve, we obtain proofs (also computer-assisted) for $9$ and $10$ runners.
\end{abstract}
\maketitle

\section{Introduction}
Wills \cite{Wills_1967} first posed the Lonely Runner Conjecture in the 1960s while studying Diophantine approximation.  The conjecture has since become a well-studied problem in combinatorics, with interpretations in terms of geometric view obstruction \cite{cusick1972view} and colourings of distance graphs \cite{zhu2002circular}. The following formulation, which gives the conjecture its name, comes from \cite{bienia1998flows}: 

Consider $k+1$ runners with distinct constant speeds.  Suppose they simultaneously begin running around a circular track with unit circumference, starting from the same location, and continue indefinitely. We say that a runner is \emph{lonely} if he/she is at a distance of at least $\frac{1}{k+1}$ away from all other runners. The Lonely Runner Conjecture predicts that regardless of the runners' speeds, each runner will become lonely at some time.

To prove the conjecture, it suffices to consider the case where one runner's speed is zero and show that this stationary runner is guaranteed to become lonely. A non-trivial result of \cite{bohman2001six} shows that it suffices to consider the case where all nonzero speeds are positive integers, though this is also implicit from \cite{wills1968simultanen} and \cite{betke1972untere}. Consequently, the Lonely Runner Conjecture for $k+1$ runners is equivalent to the following statement, where $\norm{x}$ denotes the distance from a real number $x$ to $\mathbb{Z}$.

\begin{conjecture}[Lonely Runner Conjecture for $k+1$ runners]
\label{conj:lrc}
    Let $v_1,\ldots,v_k$ be positive integers. There exists $t\in\mathbb{R}$ such that for every $i\in\{1,\ldots,k\}$,
    \[\norm{tv_i}\geq\frac{1}{k+1}.\]
\end{conjecture}

At present, Conjecture \ref{conj:lrc} is known to be true for $k\leq7$ (i.e., for a total of at most $8$ runners, including the stationary runner). The first two cases $k\in\{1,2\}$ are easy, and $k=3$ was resolved by Betke and Wills \cite{betke1972untere}. The case $k=4$ was solved first by Cusick and Pomerance \cite{cusick1984view} with computer assistance, and a proof without computer assistance was given by Bienia, Goddyn, Gvozdjak, Seb{\H{o}} and Tarsi \cite{bienia1998flows}. The proof of $k=5$ was due to Bohman, Holzman, and Kleitman \cite{bohman2001six}. Renault \cite{renault2004view} later produced a simpler proof for $k\leq5$. The case $k=6$ was settled by Barajas and Serra \cite{barajas2008lonely} in 2008.  In 2025, Rosenfeld \cite{rosenfeld2025lonely} proved the Lonely Runner Conjecture for $k=7$ using a new computational approach. In this paper, we prove the Lonely Runner Conjecture for $k \in \{8,9\}$ by refining Rosenfeld's approach in \cite{rosenfeld2025lonely}. 

\begin{theorem}
\label{thm:lrc}
    The Lonely Runner Conjecture holds for 9 and 10 runners.
\end{theorem}

Note that a day after this paper was published as a preprint, Rosenfeld \cite{rosenfeld2025lonely2} himself also independently announced a proof for $k=8$, using a different refinement of \cite{rosenfeld2025lonely}.

\section{Rosenfeld's computational approach}

In this preliminary section we give some context for Rosenfeld's approach in \cite{rosenfeld2025lonely} and indicate the source of our improvement.

It has been known for a number of years that for each $k$, deciding the Lonely Runner Conjecture for up to $k+1$ runners is a finite calculation.  The first result in this direction is due to Tao \cite{tao2018some}, who showed that there is an explicitly computable constant $C>0$ such that the following holds: Suppose that the Lonely Runner Conjecture holds for $k$ runners; then the Lonely Runner Conjecture for $k+1$ runners holds unless there is a counterexample with all (positive integer) speeds at most $k^{Ck^2}$.  This means that to prove or disprove the Lonely Runner Conjecture for $k+1$ runners, it suffices to check all sets of speeds up to $k^{Ck^2}$.  Unfortunately, the constant $C$ is enormous, so this result is not practical for obtaining any new instances of the Lonely Runner Conjecture.

This finite-checking result was improved by Giri and Kravitz \cite{giri2023structure} and by Malikiosis, Santos and Schymura \cite{malikiosis2024linearly}.  The latter authors showed that, assuming that the Lonely Runner Conjecture holds for $k$ runners, the Lonely Runner Conjecture for $k+1$ runners can be proven by checking that there are no positive integer counterexamples $v_1,v_2, \ldots, v_k$ with\footnote{Indeed, the bound given in \cite{malikiosis2024linearly} is in terms of a sum, but this can be easily transformed to a product via the AM-GM inequality (see Corollary 3 of \cite{rosenfeld2025lonely}).} 
\[v_1v_2\cdots v_k<\left[\frac{\binom{k+1}{2}^{k-1}}{k}\right]^k.\]
This quantity is much more reasonable, but even for $k=7$ it is large enough that a naive case exhaustion is not feasible.

Rosenfeld's breakthrough idea was showing that in any putative counterexample $v_1, v_2, \ldots, v_k$ to the Lonely Runner Conjecture for $k+1$ runners, there must be many primes dividing the product $v_1v_2\cdots v_k$.  If one can exhibit sufficiently many primes that must divide $v_1v_2\cdots v_k$ in any counterexample, then one can conclude that any counterexample has $v_1v_2\cdots v_k$ above the threshold of Malikiosis, Santos, and Schymura.  This would imply that in fact the Lonely Runner Conjecture holds for $k+1$ runners.

Rosenfeld introduced a clever criterion for showing that a given prime $p$ must divide $v_1v_2\cdots v_k$ in any counterexample.  He achieved this by combining the ``pre-jump'' idea of \cite{bienia1998flows} with the ansatz of checking rational times with denominator $(k+1)p$.  This computation turned out to be barely manageable for $k=7$ with a computer; the last few primes each took tens of hours to be checked.  The approach is too slow to handle larger $k$.

Our main innovation is introducing a ``sieve lemma'' that makes checking for a dividing prime $p$ more efficient.  This improvement allows us to carry out Rosenfeld's general strategy for $k \in \{8,9\}$.

\section{Reduction of possible counterexamples}
\label{sec:pre}
For positive integers\footnote{In particular, $p$ does not need to be a prime, although it will be in our applications.} $\ell,p$, we define the set \[B(\ell,p)=\{0,1,\ldots,\ell p-1\}\setminus\{0,p,\ldots,(\ell-1)p\}.\] We call a $k$-tuple $\vbf=(v_1,\ldots,v_k)\in\mathbb{N}^k$ a \textit{speed tuple}. A speed tuple $\vbf$ is said to have the \emph{LR property} if it satisfies the Lonely Runner Conjecture.

We begin by tracing the steps of \cite{rosenfeld2025lonely}, but in greater generality.

\begin{definition}
    Let $k\geq2$. A speed tuple $\vbf\in B(\ell,p)^k$ is said to be $(k,\ell,p)$-\textit{proper} if one of the following holds:
    \begin{itemize}
        \item there exists $i\in\{1,\ldots,k\}$ such that \[\gcd(v_1,\ldots,v_{i-1},v_{i+1},\ldots,v_k,\ell p)>1,\]or
        \item there exists $t\in\{0,1,\ldots,\ell p-1\}$ such that, for all $i\in\{1,\ldots,k\}$,
        \[\norm*{\frac{tv_i}{\ell p}}\geq\frac{1}{k+1}.\]
    \end{itemize}
    If $\vbf\in B(\ell,p)^k$ is not $(k,\ell,p)$-proper, then we say it is $(k,\ell,p)$-\textit{improper}. Define $I(k,\ell,p)$ to be the set of all $(k,\ell,p)$-improper speed tuples.
\end{definition}

The following is a fact noted in \cite{bienia1998flows} and \cite{rosenfeld2025lonely}, which we state without proof.

\begin{lemma}[Lemma 5 of \cite{rosenfeld2025lonely}]
\label{rem:div}
    Suppose that the Lonely Runner Conjecture holds for $k-1$, where $k\geq3$. If $\vbf$ is a speed tuple with $\gcd(v_1,\ldots,v_k)=1$, and there is $i$ such that $\gcd(v_1,\ldots,v_{i-1},v_{i+1},\ldots,v_k)>1$, then $\vbf$ has the LR property.
\end{lemma}

We easily generalize Lemma 6 of \cite{rosenfeld2025lonely}, which is the special case $\ell=k+1$ of the following lemma.

\begin{lemma}
\label{lem:div}
    Let $k \geq 3$, and assume that the Lonely Runner Conjecture holds for $k-1$ runners.  Let $p$ be a prime.  If there is some $\ell$ such that $I(k,\ell,p)$ is empty, then any counterexample $\mathbf u=(u_1,\ldots, u_k)$ to the Lonely Runner Conjecture with $k$ runners must have the property that $p$ divides $u_1u_2\cdots u_k$. 
\end{lemma}
\begin{proof}
    Assume that $I(k,\ell,p)$ is empty. Suppose that $\ubf$ is a speed tuple such that $p\nmid u_1\cdots u_k$. We may assume that $\gcd(u_1,\ldots,u_k)=1$, because the LR property is invariant under rescaling and $p$ divides the original $u_1u_2\cdots u_k$ if $p$ divides the rescaled $u_1u_2\cdots u_k$.
    
    Let $v_i$ be the residue of $u_i$ modulo $\ell p$. It follows that $\vbf\in B(\ell,p)^k$ because otherwise $p\mid v_i$ for some $i$ and hence $p\mid u_i$. Thus $\vbf$ must be $(k,\ell,p)$-proper. 
    
    If there is $i$ such that $\gcd(v_1,\ldots,v_{i-1},v_{i+1},\ldots,v_k,\ell p)=d>1$, then as $d\mid \ell p$, we have \[d\mid\gcd(u_1,\ldots,u_{i-1},u_{i+1},\ldots,u_k).\] By Lemma \ref{rem:div}, $\ubf$ has the LR property.

    Alternatively, $\vbf$ is proper in the second sense, i.e., there exists $t\in\{0,1,\ldots,\ell p-1\}$ such that, for all $i\in\{1,\ldots,k\}$, $\norm{\frac{tv_i}{\ell p}}\geq\frac{1}{k+1}$. Then for all $i$, $\norm{\frac{tu_i}{\ell p}}\geq\frac{1}{k+1}$, so $\ubf$ also has the LR property. Therefore, $\ubf$ is not a counterexample to the Lonely Runner Conjecture.
\end{proof}

In theory, nothing prevents us from picking any $\ell$ to check if $I(k,\ell,p)$ is empty. The significance of $\ell=k+1$ is that there is no hope of having $I(k,\ell,p)=\emptyset$ for smaller $\ell$, when $p$ is prime. In fact, for $p>k+1$, $I(k,\ell,p)=\emptyset$ only if $p\ell$ is a multiple of $k+1$. To see this, note that the speed tuple $\ubf=(1,2,\ldots,k)$ achieves its maximum loneliness \[\max_{t\in\mathbb{R}}\min_{i\in\{1,\ldots,k\}}\norm{tu_i}=\frac{1}{k+1}\] exactly at the times $t=\frac{a}{k+1}$ for $a$ coprime to $k+1$.

Even if it may be impossible to exclude everything from $I(k,\ell,p)$ for small $\ell$, we can still obtain valuable information regarding $I(k,m,p)$ from $I(k,\ell,p)$, provided that $m=c\ell$ for some integer $c$. The following sieve lemma tells us that if we have already checked that $(v_1,\ldots,v_k)$ is $(k,\ell,p)$-proper, we can discard from $I(k,m,p)$ all $\vbf$ of the form \[(v_1+a_1\ell p,\ldots,v_k+a_k\ell p)\] where $a_1,\ldots,a_k\in\mathbb{Z}_c$. This observation will later make it easier to show computationally that $I(k,k+1,p)$ is empty.

\begin{definition}
    Let $m=c\ell$ for some positive integer $c$. The \textit{fiber} of $I(k,\ell,p)$ at level $m$, denoted $\delta_mI(k,\ell,p)$, is defined as
    \[\big\{(v_1+a_1\ell p,\ldots,v_k+a_k\ell p):\vbf\in I(k,\ell,p), a_i\in\{0,1,\ldots,c-1\}\big\},\] i.e., all $\vbf\in B(m,p)^k$ whose images modulo $\ell p$ are in $I(k,\ell,p)$.
\end{definition}

\begin{lemma}
\label{lem:sieve}
    For any $p$, if $\ell$ divides $m$, then
    \[I(k,m,p)\subseteq\delta_mI(k,\ell,p).\]
\end{lemma}
\begin{proof}
    Let $m=c\ell$. We must show that
    \[B(m,p)^k\setminus\delta_mI(k,\ell,p)\subseteq B(m,p)^k\setminus I(k,m,p).\] Let $\ubf\in B(m,p)^k\setminus\delta_mI(k,\ell,p)$. Then, we can write $u_i=v_i+a_i\ell p$ for some $a_i\in\{0,1,\ldots,c-1\}$ and $(k,\ell,p)$-proper speed tuple $\vbf$. 
    
    If there is $i\in\{1,\ldots,k\}$ such that $\gcd(v_1,\ldots,v_{i-1},v_{i+1},\ldots,v_k,\ell p)=d>1$, then $d\mid v_j$ for all $j\neq i$ and $d\mid \ell p$, so in particular $d\mid u_j$ for all $j\neq i$ and $d\mid mp$. Hence $d\mid\gcd(u_1,\ldots,u_{i-1},u_{i+1},\ldots,u_k,mp)$ and so $\ubf$ is $(k,m,p)$-proper.

    Otherwise, there exists $t\in\{0,1,\ldots,\ell p-1\}$ such that, for all $i\in\{1,\ldots,k\}$, $\norm{\frac{tv_i}{\ell p}}\geq\frac{1}{k+1}$. Let $s=ct\in\{0,1,\ldots,mp-1\}$. Then, for all $i$,
    \[\norm*{\frac{su_i}{mp}}=\norm*{\frac{ct(v_i+a_i\ell p)}{mp}}=\norm*{\frac{c tv_i}{mp}+a_it}=\norm*{\frac{ctv_i}{mp}}=\norm*{\frac{tv_i}{\ell p}}\geq\frac{1}{k+1}.\]
    Therefore, $\ubf$ is $(k,m,p)$-proper. In both cases, $\ubf\in B(m,p)^k\setminus I(k,m,p)$. 
\end{proof}

A consequence of Lemma \ref{lem:sieve} is that the proportion of improper speed tuples decreases as we make the ansatz finer. More precisely, if $k,p$ are fixed and $\ell\mid m$, then the proportion of improper speed tuples at level $m$ is at most the proportion of improper speed tuples at level $\ell$. Therefore, if we have a sequence $(\ell_n)$ where $\ell_n\mid\ell_{n+1}$, then this proportion at level $\ell_n$ is non-increasing.

\section{Verification of dividing primes}
\label{sec:ver}
From Lemma \ref{lem:div}, if we can show that $I(k,k+1,p)$ is empty, then we have that $p\mid v_1v_2\ldots v_k$ whenever $\vbf$ does not have the LR property. Following \cite{rosenfeld2025lonely}, we wish to find many primes $p$ for which $I(k,k+1,p)$ is empty. Consider a fixed prime $p$.

However, unlike \cite{rosenfeld2025lonely}, instead of directly checking all possible $(p-1)^k(k+1)^k$ speed tuples in $B(k+1,p)^k$, we first use Lemma \ref{lem:sieve} to ``sieve out" most candidates. More precisely, we compute $I(k,\ell,p)$ for some smaller $\ell\mid (k+1)$ first, and the compute $I(k,k+1,p)$ inside the fiber of $I(k,\ell,p)$. Furthermore, we can choose $\ell_1,\ldots,\ell_n$ such that $\ell_1=1,\ell_n=k+1$ and $\ell_{r+1}=c_r\ell_r$ for some $c_r\in\mathbb{N}$ for $r\in\{1,\ldots,n-1\}$. At the end, we have to check only
\[(p-1)^k\left(\ell_1^k+s_1\ell_2^k+s_1s_2\ell_3^k+\cdots+s_1\cdots s_{n-1}\ell_n^k\right)\]
tuples, where $0\leq s_r\leq1$ is the proportion of tuples that survive the sieve at level $\ell_r$. This number can be made even smaller if we use two distinct sequences $\ell_1,\ldots,\ell_n$ and $\ell'_1,\ldots,\ell'_{n'}$ that merge (by intersection) at the end, as illustrated by the case $k+1=10$. 

In the worst case where $s_1=\cdots=s_{n-1}=1$, this number is actually slightly larger than $(k+1)^k(p-1)^k$ (because we wasted time checking prior to the last step). In practice, however, $s_r$ tends to be very small and so this improves on the direct method.

The method we described applies to any value of $k$, but it is most effective when $k+1$ is highly composite. For the sake of illustration and completeness, we explicitly state the algorithms for the cases $k=8,9$, which are relevant to proving the Lonely Runner Conjecture for 9 and 10 runners.

\subsection{Case $k=8$}
\begin{enumerate}
    \item Compute $I(8,1,p)$. Then find its fiber $\delta_3I(8,1,p)$.
    \item Compute $I(8,3,p)$ within $\delta_3I(8,1,p)$. Again, find the fiber $\delta_9I(8,3,p)$.
    \item Compute $I(8,9,p)$ within $\delta_9I(8,3,p)$.
\end{enumerate}

\subsection{Case $k=9$}
\begin{enumerate}
    \item Compute $I(9,1,p)$. Then find its fibers $\delta_2I(9,1,p)$ and $\delta_5I(9,1,p)$.
    \item 
    \begin{enumerate}
        \item Compute $I(9,2,p)$ within $\delta_2I(9,1,p)$, and find $\delta_{10}I(9,2,p)$.
        \item Compute $I(9,5,p)$ within $\delta_5I(9,1,p)$, and find $\delta_{10}I(9,5,p)$.
    \end{enumerate}
    \item Compute $I(9,10,p)$ within $\delta_{10}I(9,2,p)\cap\delta_{10}I(9,5,p)$.
\end{enumerate}

\section{Proof of Theorem \ref{thm:lrc}}
\label{sec:main}
Let $C_k=\left[\frac{\binom{k+1}{2}^{k-1}}{k}\right]^k$ be the quantity derived from the result of Malikiosis, Santos and Schymura \cite{malikiosis2024linearly}. Notice that $C_8<10^{80}$ and $C_9<10^{111}$.

Let
\[\begin{split}
        S_8=\{&47, 53, 59, 61, 67, 71, 73, 79, 83, 89, 97, 101, 103, 107, 109,\\&113, 127, 131, 137, 139, 149, 151, 157, 163, 167, 173, 179,\\&181, 191, 193, 197, 199, 211, 223, 227, 229, 233, 239, 241\}
    \end{split}\]
and
\[\begin{split}
        S_9= \{&137, 139, 149, 151, 157, 163, 167, 173, 179, 181, 191, 193,\\& 197, 199, 211, 223, 227, 229, 233, 239, 241, 251, 257, 263,\\& 269, 271, 277, 281, 283, 293, 307, 311, 313, 317, 331, 337,\\& 347, 349, 353, 359, 367, 373, 379, 383, 389, 397, 401\}.
\end{split}\]
Using the algorithms presented in Section \ref{sec:ver}, we found that $I(k,k+1,p)$ is the empty set, for every $p\in S_k$ and $k\in\{8,9\}$. 

As the products satisfy
\[\prod_{p\in S_8}p>10^{82}\]
and
\[\prod_{p\in S_9}p>10^{112},\] it follows that the Lonely Runner Conjecture holds for $k=8$ and (then) $k=9$.

\section{Implementation details}
We carry out the verification algorithms in C++. Our code is developed based on that of Rosenfeld \cite{codebase}, but we optimize for parallelization across cores. We employed OpenAI GPT-5 to assist with code generation, especially in low-level optimization. The code used, along with the result receipts, is uploaded to the author's GitHub \cite{mycode}. The computations were performed on a 14-core Apple M3 Max processor. On this machine, the verification for $k=8$ takes 15 minutes, and for $k=9$ takes under 23 hours.

We remark that we work with sets instead of $k$-tuples (to save a factor of $k!$).  

For $k=8$, we tested every prime from\footnote{For $p\leq k+1$, in Lemma \ref{lem:div}, $t=1$ serves as a witness time for every $\vbf\in B(k+1,p)^k$, so trivially there is no $(k,k+1,p)$-improper speed tuple. Had we exploited this fact in Section \ref{sec:main}, we would have got a factor of $\operatorname{lcm}(1,2,\ldots,k)$ for free, and needed to check a couple fewer primes.} $11$ to $241$.  For all $11\leq p\leq43$, the set $I(8,9,p)$ was non-empty. On the other hand, for $k=9$, we only checked primes from $137$ to $401$. It should not be inferred from our prior discussion that $I(9,10,p)$ is non-empty for $11\leq p\leq131$.

There are two reasons why we start $p$ from a large value. One is that larger $p$ contributes more to the product. But more importantly, here it is actually faster to verify large $p$ (in the range of low hundreds). The running time forms a U-shaped curve with run time of 1-2 hours from the start\footnote{$p=151$ is a far outlier here, taking 7.5 hours to verify.}, going as low as 5 minutes for $227\leq p\leq277$, and going up again to 40 minutes for the last $p$.

The surprising behavior of running time reflects the trade-off between the size of $B(1,p)^9$ and the size of $I(9,1,p)$ in the first step. When $p$ is small, there are not many $\vbf\in B(1,p)^9$ to check, but it is much more likely (because the ansatz in Lemma \ref{lem:div} allows fewer witness times) that $\vbf$ belongs in $I(9,1,p)$, so we need more time in steps 2 and 3 to check $\delta_mI(9,1,p)$. The opposite is true for $p$ large. In the middle, there exists a ``Goldilocks zone" of primes such that both $I(9,1,p)$ and $B(1,p)^9$ are not too large, and in this case it happens to be around 250.

\section{Further work}
Since $11$ is prime, our sieve algorithm in the case $k=10$ has only two steps: find $I(10,1,p)$ first, then $I(10,11,p)$. In contrast to when $k+1$ is composite, we do not have an intermediate sieve. Each $\vbf\in I(10,1,p)$ corresponds to $11^{10}$ speed tuples in $\delta_{11}I(10,1,p)$, so this quickly becomes the bottleneck that makes verifying $p$ impossible in a reasonable time.

However, since we only want to conclude that $I(10,\ell,p)$ is empty, we can also insert an additional filter, say 2, and then compute $I(10,22,p)$ within $\delta_{11}I(10,2,p)$. This additional filter may sufficiently reduce $I(10,2,p)$ to a manageable number of cases. Still, within this framework we cannot avoid extending from $\ell$ to $11\ell$ at some point because, as we previously discussed, for $I(10,\ell,p)$ to be empty, we must have $11\mid\ell$. 

Proving the case $k=11$ assuming the case $k=10$ seems more approachable. By our estimate, proving this using our refinement alone will take on the order of weeks. Moreover, Rosenfeld's \cite{rosenfeld2025lonely2} refinement on the gcd condition can be combined with ours to further reduce the number of candidate speed tuples. 

Ultimately, our algorithm, like the original proposed by Rosenfeld, still needs to perform $\Theta(p^k)$ checks to verify that $I(k,\ell,p)$ is empty. Our reduction makes the constant of the scaling much smaller, but we see no easy way to escape this scaling law.

\section*{Acknowledgements}
I am very grateful to my adviser, Noah Kravitz, for his thoughtful guidance, constant encouragement and meticulous feedback that has vastly improved the paper. I thank the referees for their valuable comments. I am supported by the King's Scholarship (Thailand).

\bibliographystyle{plain}
\bibliography{refs}

\end{document}